\documentclass[12pt]{article}

\usepackage[all]{xy}
\usepackage{graphicx}
\usepackage{amssymb,latexsym,amsmath,amsthm, mathrsfs}
\usepackage{fullpage}

\usepackage{scalerel,stackengine}
\newtheorem{theorem}{Theorem}[section]
\newtheorem{lemma}[theorem]{Lemma}

\stackMath
\newcommand\reallywidehat[1]{%
\savestack{\tmpbox}{\stretchto{%
  \scaleto{%
    \scalerel*[\widthof{\ensuremath{#1}}]{\kern-.6pt\bigwedge\kern-.6pt}%
    {\rule[-\textheight/2]{1ex}{\textheight}}
  }{\textheight}%
}{0.5ex}}%
\stackon[1pt]{#1}{\tmpbox}%
}
\newtheorem{definition}[theorem]{Definition}

\newtheorem{proposition}[theorem]{Proposition}
\newtheorem{corollary}[theorem]{Corollary}

\newenvironment{acknowledgement}[1][Acknowledgements]
{\begin{trivlist} \item[\hskip \labelsep {\bfseries #1}]}
{\end{trivlist}}

\newtheorem{remark}[theorem]{Remark}
 \DeclareMathSymbol{\N}{\mathbin}{AMSb}{"4E}

\DeclareMathSymbol{\Z}{\mathbin}{AMSb}{"5A}
\DeclareMathSymbol{\R}{\mathbin}{AMSb}{"52}
\DeclareMathSymbol{\Q}{\mathbin}{AMSb}{"51}
\DeclareMathSymbol{\I}{\mathbin}{AMSb}{"49}
\DeclareMathSymbol{\C}{\mathbin}{AMSb}{"43}

\def\M{{\mathbb M}}

\numberwithin{equation}{section}

\title{A characterization of real holomorphic chains and applications in representing homology classes by algebraic cycles}
\author{Jyh-Haur Teh, Chin-Jui Yang}
\date{}

\begin{document}
\maketitle

\begin{abstract}
We show that a $2k$-current $T$ on a complex manifold is a real holomorphic $k$-chain if and only if $T$ is locally real rectifiable, $d$-closed and has $\mathcal{H}^{2k}$-locally finite support. This result is applied to study homology classes represented by algebraic cycles.
\end{abstract}

\maketitle
\section{Introduction}
A holomorphic chain on a complex manifold is a formal finite linear combination of some irreducible holomorphic subvarieties with integer coefficients. Since every holomorphic variety $V$ of dimension $k$ naturally defines a $d$-closed integral current $[V]$ of type $(k, k)$, it is a natural question to ask if the three properties: $d$-closed, integral and type $(k, k)$ suffice to characterize holomorphic chains. For $d$-closed positive integral currents, the problem was solved by King (\cite{K71}). For general $d$-closed integral currents of type $(k, k)$, the problem was solved by Harvey and Shiffman (\cite{HS74}) but with a hypothesis on the support of currents. Harvey and Shiffman conjectured that this hypothesis was not needed, but
they were unable to overcome it. This problem was finally solved by Alexander (\cite{A97}) after more than 20 years later. The case for holomorphic chains with real coefficients is quite different from the integral case. Being $d$-closed, type $(k, k)$ and real rectifiable may not be holomorphic chains with real coefficients. In this case, we really need
restriction on the support of currents. We solved this problem for positive $d$-closed real rectifiable currents in \cite{JC}, and in this paper, we solve this problem completely without the positivity condition. The following is our main result.

\begin{theorem}\label{main theorem}
Let $X$ be a complex manifold. A $2k$-current $T$ is a real holomorphic chain on $X$ if and only if $T \in RR^{loc}_{k,k}(X)$ is $d$-closed and $\mbox{spt}(T)$ is $\mathcal{H}^{2k}$-locally finite.
\end{theorem}

From this result, we are able to generalize the structure theorem of Harvey and King (Theorem \ref{structure theorem}) and some results about stable currents of Harvey
and Shiffman.
We are interested in this problem not only because it is a natural question to ask, but it also has some applications in our study of algebraic cycles, tightly related
to the Hodge conjecture. For example in Proposition \ref{R+dd^c}, we prove that on a complex projective manifold,
if a $d$-closed smooth form $e$ considered as a current can be written as
$$e=P+da$$
where $P$ is a Lipschitz $2k$-chain with rational coefficients which is $d^c$-closed, then $e$ is homologous to some algebraic cycle with rational coefficients. Application of methods from geometric measure theory to study algebraic cycles is very fruitful, some papers that are especially inspiring us are
\cite{D, DL03, H77, HL75, HK, HK2, K71, L75, L89}.

This paper is organized as follows. We follow Alexander's strategy to prove our main result by induction. Section 2 contains some preliminary results that are needed
to prove the case $k=1$, and Section 3 completes the induction. Our new key idea is an observation made by the second author that the positive current $T'$ associated
to a $d$-closed real rectifiable current $T$ of type $(k, k)$ with $\mathcal{H}^{2k}$-locally finite support is actually $d$-closed for $k\geq 2$. Therefore by our result for the positive case in \cite{JC}, $T'$ and hence $T$ are holomorphic chains with real coefficients. This in some sense simplifies half of Alexander's proof, but since we use Siu's semicontinuity theorem
in an essential way for the positive case, this does not mean our proof is easier, but probably easier conceptually. In Section 4, we use our main result to generalize
some results that we proved in \cite{JC} and some results in \cite{HS74}. These include some results about homologically volume minimizing currents, stable currents and
stationary currents.

\begin{acknowledgement}
We would like to thank the reviewer for his/her valuable comments on the paper and
Professor Harvey for his encouragement.
\end{acknowledgement}

\section{Preliminary results}
We fix some notations that will be used throughout this paper. Let $M$ be an oriented smooth manifold.  Let $A^r(M)$ be the space of complex-valued smooth $r$-forms on $M$ and let $A_c^r(M)$ be the space of complex-valued $r$-forms with compact support on $M$. Dually,
$\mathscr{D}^\prime_r(M)$ is the space of currents of dimension $r$ and $\mathscr{E}^\prime_r(M)$ is the space of currents of dimension $r$ with compact support.

\begin{definition}
An $r$-current $T\in \mathcal{E}'_r(M)$ with support in a compact set $K\subset M$ is said to be a real rectifiable current in $K$ if for any $\varepsilon>0$, there
is an open subset $U$ of some $\mathbb{R}^n$, a Lipschitz map $f:U \rightarrow M$ and a finite real polyhedral $r$-chain $P$ (in this article, we assume that simplices are nonoverlapping) with $f(\mbox{spt} P) \subset K$ and
$$\M(T - f_*(P))< \epsilon$$
where $\M$ denotes the mass norm.
The group of real rectifiable $r$-currents in $K$ is denoted by $RR_{r, K}(M)$ and elements of the union
$$ RR_{r}(M):=\bigcup_{\underset{\mbox{K compact}}{K\subset M,}}RR_{r, K}(M)$$
are called real rectifiable $r$-currents in $M$.
The group $RR^{loc}_r(M)$ of locally real rectifiable $r$-currents in $M$ is the collection of all
$T \in \mathscr{D}^\prime_r(M)$ such that for each $x\in M$,  there is $S \in RR_r(M)$ such that $x\notin \mbox{spt}(T - S)$.
\end{definition}

We recall the definition of real holomorphic chains.
\begin{definition}
Let $X$ be a complex manifold. A current $T \in \mathscr{D}^{\prime}_{2k}(X)$ is said to be a real holomorphic $k$-chain on
$X$ if $T$ can be written in the form $T = \sum^{\infty}_{j=1} r_j [V_j]$ where $r_j \in \mathbb{R}$ and $V = \bigcup^{\infty}_{j=1} V_j$ is a
purely $k$-dimensional holomorphic subvariety of $X$ with irreducible components $\{V_j\}^{\infty}_{j=1}$. The vector space of
real holomorphic $k$-chains on $X$ is denoted by $\mathscr{RZ}_k(X)$. Also let $\mathscr{RZ}^{+}_k(X)$ denote the
set of positive real holomorphic $k$-chains on $X$, i.e., those real holomorphic $k$-chains with nonnegative coefficients.
\end{definition}

We denote the Hausdorff $k$-measure by $\mathcal{H}^k$.

\begin{definition}
A $\mathcal{H}^k$-measurable subset $A \subset M$ is called $\mathcal{H}^k$-locally finite if for all $a \in M$, there is an open neighborhood $V\subset M$ of $a$ such that $\mathcal{H}^k(V\cap A) < \infty$.
\end{definition}

We refer the reader to \cite[Theorem 2.4.4]{K71} for the following result.
\begin{lemma}
Let $A$ be a subset of $\mathbb{C}^n$ and let $\epsilon > 0$. If $\mathcal{H}^{2k+\epsilon}(A) = 0$, then for almost all complex $(n-k)$ planes $L$ through $0$, $\mathcal{H}^{\epsilon}(A\cap L) = 0.$ When $A$ is closed in a neighborhood of $0$ and $\pi_L$ is a linear map from $\mathbb{C}^n$ to $\mathbb{C}^k$ with kernel $L$, $\mathcal{H}^1(A\cap L) = 0$ implies that there is a neighborhood $V$ of $0$ such that the restriction $\pi |_{A\cap V}$ is a proper map.
\end{lemma}

Since Theorem \ref{main theorem} is a local result, we may assume that $T\in RR^{loc}_{k, k}(U)$ for some open set $U\subset \C^n$, $dT=0$ and the support $\mbox{spt}(T)$ of $T$
is of $\mathcal{H}^{2k}$-locally finite. Furthermore, by making a translation, we may assume $0\in \mbox{spt}(T)$. This $T$ is fixed through section 2 and 3.

For $I = \{i_1,..., i_k\} \subset \{1,2,...,n\}$ where $i_1 <\cdots < i_k$, we let $\pi_I(z_1, ..., z_n)=(z_{i_1}, ..., z_{i_k})$ be the projection from $\mathbb{C}^n$ to $\mathbb{C}^k$. It follows from the above Lemma that, after a possible linear change of coordinates, each of the planes $\{z : \pi_I(z) = 0\}$ meets $X$ in a discrete set at $0$. Fix $I_0 = \{1,2,...,k\}$ and write $\pi$ for $\pi_{I_0}$. There exist $r > 0$ and $\delta > 0$ such that $\mbox{spt}(T)$ is disjoint from $\overline{\Delta^{\prime}(r)} \times \partial \Delta^{\prime \prime}(\delta)$ where
$$\Delta^{\prime}(r) = \{z^{\prime}=(z'_1, ..., z'_k) \in \mathbb{C}^k : |z_j^{\prime}| < r, 1\leq j \leq k\}$$ and
$$\Delta^{\prime \prime}(\delta) = \{z^{\prime \prime}=(z''_1, ..., z''_{n-k}) \in \mathbb{C}^{n-k} : |z^{\prime \prime}_j| < \delta, 1 \leq j \leq n-k\}$$
Set $$\Delta = \Delta^{\prime}(r) \times \Delta^{\prime \prime}(\delta)$$
and $X_0 = \mbox{spt}(T) \cap \Delta$. Then $\pi |_{X_0}$ is a proper map from $X_0$ onto $\Delta^{\prime}$ and, writing $z = (z^{\prime} , z^{\prime \prime})$, the projection $\pi$ to $\mathbb{C}^k$ is given by $\pi(z) = z^{\prime}$. Also, let $\pi_j : \Delta \rightarrow \mathbb{C}$ where $\pi_j(z) = z_j$ be the projection to the $j$-th coordinate.

We will prove our main result by induction. First, we use the method in $\cite{A97}$ to prove the case $k = 1$ .
Now let $k = 1$ and $\rho(z) = |z_1|$. Let $T(r) = T|_{U\cap (\Delta^{\prime}(r) \times \Delta^{\prime \prime}(\delta))}$ be the restriction of $T$ to $U\cap (\Delta^{\prime}(r) \times \Delta^{\prime \prime}(\delta))$.
Shrinking $r$, we may assume the slice $\langle T|_{U\cap (\mathbb{C} \times \Delta^{\prime \prime}(\delta))} , \rho , r  \rangle$ exists as a real rectifiable 1-current supported in
$$K:= \mbox{spt}(T)\cap (\partial\Delta^{\prime}(r)\times \Delta^{\prime \prime}(\delta))$$ and note that, in this case, the slice
$$\langle T|_{U\cap (\mathbb{C} \times \Delta^{\prime \prime}(\delta))} , \rho , r  \rangle = bT(r)$$ by $\cite[4.2.1]{F69}$.

From \cite[2.10.25]{F69}, we have a bound on the measure of slices cut by a Lipschitz map.

\begin{theorem}\label{volume}
If $f : \mathbb{R}^k \rightarrow \mathbb{R}^l$ is a Lipschitz map , $A \subset \R^k, \ 0 \leq m < \infty,$ and $0 \leq n < \infty$, then
$$
\int_{\mathbb{R}^l}^{\ast} \mathcal{H}^m(A \cap f^{-1}(y)) d\mathcal{H}^n(y) \leq (Lip(f))^n \frac{\Omega(m)\Omega(n)}{\Omega(m+n)}\mathcal{H}^{m+n}(A)
$$
\end{theorem}

By the above theorem, we have
$$
\int_{B_{r_0}(0)}^{\ast} \mathcal{H}^1(X \cap \rho^{-1}(r)) d\mathcal{L}^1(r) \leq 4\mathcal{H}^2(X)
$$
for any $r_0 > 0.$ So we may assume $\mathcal{H}^1(K) < \infty$.

Since $bT(r)$ is a closed real rectifiable current, the set
$$N:= \{ z \in \mathbb{C}^n : \Theta^1 (bT(r),z) > 0 \}$$ is countably 1-rectifiable. So there exist countably many $C^1$-curves $\{\gamma_j\}^{\infty}_{j=1}$ such that for each $j$, there is a $\mathcal{H}^1$-integrable function $\theta_j$ on $\gamma_j$ which satisfy
$$
\theta_j(z) =
\left\{ \begin{array}{ll}
\Theta^1 (bT(r),z), & \textrm{if $z \in N$}\\
0, & \textrm{otherwise}\\
\end{array} \right.
 \ and \ \
 bT(r) = \sum^{\infty}_{j=1}\gamma_j\wedge\theta_j. \
$$
Hence
$$
\M(bT(r)) = \sum^{\infty}_{j=1} \M(\gamma_j\wedge\theta_j) = \sum^{\infty}_{j=1} \int_{\gamma_j} |\theta_j|d\mathcal{H}^1 < \infty.
$$
Also, $\|bT(r)\| = \sum^{\infty}_{j=1}\mathcal{H}^1\Bigl\lfloor|\theta_j|$. Here, $\gamma_j = image (c_j),$ for some $C^1$-embedding $c_j : (0,1) \rightarrow U$, and the orientation of $\gamma_j$ is induced from the natural orientation of $(0,1)$.\\

\begin{definition}
Let $\gamma : (0,1) \rightarrow \mathbb{C}^n$ be a $C^1$-curve  and $f$ be a $\mathcal{H}^1$-integrable complex-valued function on $\gamma((0,1))$. Write $\gamma(t) = (\gamma_1(t) , ... , \gamma_n(t)).$
We define
$$
\int_{\gamma} |f||dz_1| \equiv \int_{0}^{1} |f\circ\gamma(t)||\gamma_1^{\prime}(t)| dt.
$$
\end{definition}

By a simple computation, we have :

\begin{proposition}\label{Cauchy inequality}
Suppose $\gamma : (0,1) \rightarrow \mathbb{C}^n$ is a $C^1$-curve  and $f$ is a $\mathcal{H}^1$-integrable complex-valued function on $\gamma((0,1))$. Then
$$|\int_{\gamma} f dz_1| \leq \int_{\gamma} |f||dz_1|.$$
\end{proposition}

From \cite[section 12]{L83}, we have an area formula.

\begin{theorem}\label{area formula}(Area formula)
Let $A$ be a $\mathcal{H}^n$-measurable and countably $n$-rectifiable set in $\mathbb{R}^{n+k}$ and $f$ be a locally Lipschitz map on $V$ into $\mathbb{R}^{n + k_1}$ where $V$ is an open set in $\mathbb{R}^{n+k}$ containing $A$. If $g$ is a non-negative $\mathcal{H}^n$-measurable function on $A$, then
$$
\int_A g \ J_Af d\mathcal{H}^n = \int_{\mathbb{R}^{n+k_1}} \{\int_{A\cap f^{-1}(y)} g d\mathcal{H}^0\} d\mathcal{H}^{n}(y).
$$
\end{theorem}

\begin{proposition}\label{cauchy inequality for real rectifiable 1-current}
Let $\theta$ be a $\mathcal{H}^1$-integrable real-valued function on $\gamma$ where $\gamma : (0,1) \rightarrow \mathbb{C}^n$ is a $C^1$-curve  and consider the real rectifiable 1-current $\gamma\wedge\theta$ where the orientation of $\gamma$ is induced from the natural orientation of (0,1).
We have
$$
\int_{\gamma} |\theta||dz_1| \leq \M(\gamma\wedge\theta).
$$
\end{proposition}

\begin{proof}
\begin{align*}
\int_{\gamma} |\theta||dz_1|  & = \int_{0}^{1} |\theta\circ\gamma(t)||\gamma_1^{\prime}(t)| \ dt= \int_{0}^{1} |\theta\circ\gamma(t)||(\pi_1\circ\gamma)^{\prime}(t)| \ dt \\
& = \int_{\R^2} \{ \int_{(\pi_1\circ\gamma)^{-1}(y)} |\theta\circ\gamma| \ d\mathcal{H}^0 \} \ d\mathcal{H}^1(y) \ (by \ Theorem \ \ref{area formula}) \\
& = \int_{\pi_1(\gamma)} \{ \int_{(\pi_1\circ\gamma)^{-1}(y)} |\theta\circ\gamma| \ d\mathcal{H}^0 \} \ d\mathcal{H}^1(y) \\
& = \M(\pi_{1\ast}(\gamma , |\theta|)) \ (Here \ (\gamma , |\theta|) \ is \ a \ varifold, see \ \cite[15.6]{L83}) \\
& = \int_{\gamma} (J_{\gamma}\pi_1) |\theta| \ d\mathcal{H}^1 \ (by \ Theorem \ \ref{area formula}, see \ \cite[15.7]{L83})\\
& \leq \int_{\gamma} |\theta| \ d\mathcal{H}^1 = \M(\gamma\wedge\theta).
\end{align*}
\end{proof}

We now consider the Cauchy transform of the 1-current $bT(r)$ in the coordinate function $z_1$ (see \cite[pg 8]{W}).
\begin{lemma}\label{bT(r) finite}
Define
$$
bT(r)(\frac{|dz_1|}{|z_1 - \alpha|}):= \sum^{\infty}_{j=1} \int_{\gamma_j} \frac{|\theta_j||dz_1|}{|z_1 - \alpha|}
$$
Then $bT(r)<\infty$ for $\mathcal{L}^2$-a.e. $\alpha$ in $\mathbb{C}$.
\end{lemma}

\begin{proof}
Fix $R >0$ with $\pi(spt(bT(r))) \subset \pi(K) \subset B_R(0)$, by Fubini's theorem,
$$
\int_{|z_1| < R} \large\{ \int_{\gamma_j} \frac{|\theta_j||dz_1|}{|z_1 - \alpha|} \} dx\wedge dy =
\int_{\gamma_j} |\theta_j| \int_{|z_1| < R} \frac{dx\wedge dy}{|z_1 - \alpha|}|dz_1|
$$
For $\alpha \in \pi(spt(bT(r)))$ and $|z_1| < R$, $|z_1 - \alpha| < 2R$, then by making $z_1'=z_1-\alpha$, we have
$$
\int_{|z_1| < R} \frac{dx\wedge dy}{|z_1 - \alpha|} \leq \int_{|z_1^{\prime}| < 2R} \frac{dx^{\prime}\wedge dy^{\prime}}{|z_1^{\prime}|} = \int_{0}^{2R} rdr \int_{0}^{2\pi} \frac{d\theta}{r} = 4\pi R
$$
Hence by Proposition \ref{cauchy inequality for real rectifiable 1-current}(see \cite[Lemma 2.4]{W}),
$$
\int_{|z_1| < R} \large\{ \int_{\gamma_j} \frac{|\theta_j||dz_1|}{|z_1 - \alpha|} \} dxdy \leq 4\pi R \cdot \int_{\gamma_j}|\theta_j||dz_1|  \leq 4\pi R \cdot \M(\gamma_j\wedge\theta_j) < \infty
$$

\end{proof}

For almost all $\alpha$, it is a Federer's result that the slice $\langle T(r) , \pi , \alpha \rangle$ exists and is a real rectifiable 0-current, i.e., this slice is given by $\sum_{j=1}^{s} n_j[w_j]$ where  $w_j$'s are distinct points in $\Delta$ with $\pi(w_j) = \alpha$ and $n_j$'s are non-zero real numbers (see \cite[4.3.8]{F69}). Alexander proved a Cauchy formula in \cite[pg 125]{A97} for $bT(r)$ when $T$ is an integral current. We observe that his proof is actually valid for locally real rectifiable currents of type $(1, 1)$. We state his result in a more general form as follows.

\begin{theorem}(Alexander's Cauchy Formula)
Let $U\subset \C^n$ be an open set and $T\in RR^{loc}_{1, 1}(U)$ be a closed locally real rectifiable current of type $(1, 1)$ on $U$.
Fix $r, \delta>0$. Let $T(r):=T|_{U\cap(\Delta'(r)\times \Delta''(\delta))}$ and $\pi:\C^n\rightarrow \C$ be the projection $\pi(z_1, ..., z_n)=z_1$.
Fix $\alpha\in \C$. Suppose that the slice $bT(r)$ exists as a closed locally real rectifiable 1-current supported on $\partial(\Delta'(r)\times \Delta''(\delta))$,
\begin{enumerate}
\item
$bT(r)(\frac{|dz_1|}{|z_1 - \alpha|}) < \infty$
and
\item
the slice $\ \langle T(r) , \pi , \alpha \rangle$ exists and is equal to $\sum_{j=1}^{s} n_j[w_j]$ where
all $n_j$'s $\in \R$,
\end{enumerate}
then for all holomorphic functions $f$ in $\mathbb{C}^n$,
$$
bT(r)(\frac{fdz_1}{z_1 - \alpha}) = 2\pi i \langle T(r) , \pi , \alpha \rangle(f) = 2\pi i \sum_{j=1}^{s} n_jf(w_j).
$$
\end{theorem}

\begin{remark}
The integral representing $bT(r)(\frac{fdz_1}{z_1 - \alpha})$ converges (absolutely) by Proposition $\ref{Cauchy inequality}$ and 1.
\end{remark}

\begin{lemma}
Suppose that $\alpha$ satisfies the hypothesis of the Alexander's Cauchy formula and the slice $\langle T(r) , \pi , \alpha \rangle = \sum_{j=1}^{s} n_j[w_j]$ is non-zero. Set $w = w_1$. Then there is a representing measure $\mu$ for $w$ for the uniform algebra $\mathcal{A}$ of functions holomorphic on $\Delta$ and continuous on $\Bar{\Delta}$ with $\mbox{supp}(\mu) \subset \partial\Delta$ and $\mu$ is concentrated on $N$, i.e., $\mu(N) = 1$. In particular, $\mbox{supp}(\mu) \subset K$.
\end{lemma}

\begin{proof}
Choose a polynomial $g$ such that $g(w_1) = 1$ and $g(w_j) = 0$ for $j > 1$. Then for every polynomial $f$, by the Alexander's Cauchy formula, we have
$$
bT(r)(\frac{fgdz_1}{z_1 - \alpha}) = 2\pi i \sum^{\infty}_{j=1} n_jf(w_j)g(w_j) = 2\pi in_1f(w).
$$
Hence $f(w) = \int_N fd\sigma$ where
$$
d\sigma = \frac{1}{2\pi in_1} \sum^{\infty}_{j=1} \frac{\theta_jgdz_1}{z_1 - \alpha} |_{\gamma_j}.
$$
For any $h \in \mathcal{A}$ and $n\in \N$, the function $h_n(z):= h((1-\frac{1}{n})z)$ is holomorphic on a neighborhood of $\Bar{\Delta}$ for all $n \in \mathbb{N}$. So each $h_n$ can be approximated uniformly by polynomials on $\Bar{\Delta}$. Since $\{h_n\}^{\infty}_{n=1}$ converges to $h$ uniformly on $\Bar{\Delta}$, by diagonal process and triangle inequality, $h$ can be approximated uniformly by polynomials on $\Bar{\Delta}$.
Thus, $\sigma$ is a complex representing measure for $w$ concentrated on $N$. By \cite[Theorem \uppercase\expandafter{\romannumeral 2}.2.2]{G}, there exists a non-negative representing measure $\mu$ for $w$ which is absolutely continuous with respect to $\sigma$. Also, $\mu(N) = 1$ (see \cite[Lemma 6]{A97}).

\end{proof}

Let $\Delta(a,\lambda)$ be the open polydisc in $\mathbb{C}^n$ with center $a \in \mathbb{C}^n$ and radius $\lambda > 0$, and let $\omega$ be the standard K\"ahler form on $\mathbb{C}^n$.

\begin{lemma}\label{nonzero}
For a non-zero $R \in RR^{loc}_{k,k}(U)$ with $0 \in spt(R)$, there exists a $j$ satisfying $1 \leq j \leq n$ and a measurable set $E \subset \mathbb{C}$ of positive $\mathcal{L}^2$-measure such that the slice $\langle R|_{\Delta(0,\frac{\delta}{4})}, \pi_j , \alpha \rangle$ exists and is non-zero for all $\alpha \in E$.
\end{lemma}

\begin{proof}
See \cite[Lemma 7, Lemma 11]{A97}. The proof mainly applies \cite[4.3.8]{F69} and \cite[4.3.2 (1)]{F69} which are also true for locally real rectifiable currents.
\end{proof}

Let $\widehat{A}$ denote the polynomial convex hull of a set $A\subset \C^N$. The following result is from \cite[Lemma 2]{A1}.
\begin{lemma}\label{A1}
Let $D$ be a closed Jordan domain in $\mathbb{C}$ with rectifiable boundary, $K$ a compact subset of $\partial D$ of positive linear measure, $Q$ a polynomial convex set in $\mathbb{C}^n$, $f$ a polynomial in $\mathbb{C}^n$, and $s$ a positive integer. Assume that $Q = \reallywidehat{\ (f^{-1}(\partial D)\cap Q) \ }$ and $f|_{Q}$ is at most $s-to-1$ over points of $K$ (i.e., if $\lambda \in K$, then $f^{-1}(\lambda)\cap Q$ has at most $s$ points). Then $f^{-1}(int(D))\cap Q$ is a (possibly empty) pure 1-dimensional holomorphic subvariety of $f^{-1}(int(D))$.
\end{lemma}

\begin{lemma}\label{1}
For coordinates and the polydisc $\Delta$ chosen as in section 1 for $k = 1$, there exists an 1-dimensional subvariety $V$ of $\Delta$ such that for $\mathcal{L}^2$-a.e. $\alpha \in \{ \lambda \in \mathbb{C} : |\lambda| < r \}$, the slice $\langle T(r) , \pi , \alpha \rangle$ exists, and is a real holomorphic 0-chain in $\Delta$ with support in $V$.
\end{lemma}

\begin{proof}
This lemma is a generalization of \cite[Lemma 12]{A97} for $k = 1$ and the proof is similar to Alexander's. Assuming that $\langle T(r) , \pi , \alpha \rangle$ exists. It follows from the Alexander's Cauchy formula that there is a representing measure supported in $K$ for each $z \in \mbox{spt}(\langle T(r) , \pi , \alpha \rangle)$. Suppose that $\mu$ is a representing measure for $z$. If there is some polynomial $P$ such that $|P(z)| > \mbox{sup}_K |P|$, then
$$
|P(z)| = |\int_K P d\mu| \leq \int_K |P| d\mu < \int_K |P(z)| d\mu = |P(z)|
$$
which is a contradiction. Hence $z \in \widehat{K}$.

Since $\mathcal{H}^1(K) < \infty$, by Theorem \ref{volume} with $f$ replaced by $\pi$, there exist a set $Q$ of positive measure in $\partial\Delta^{\prime}(r)$ and a positive integer $s$ such that $\pi$ maps exactly $s$ points of $K$ to each point of $Q$. We can assume that $Q$ is compact. Since
$$
\widehat{K} = \reallywidehat{ \{[\pi^{-1}(\partial\Delta^{\prime}(r))]\cap K }\} \subset \reallywidehat{ \{ [\pi^{-1}(\partial\Delta^{\prime}(r))]\cap \widehat{K} \}} \subset \widehat{\widehat{K}} = \widehat{K},
$$
we have
$$
\widehat{K} = \reallywidehat{ \{[\pi^{-1}(\partial\Delta^{\prime}(r))]\cap \widehat{K}\} }.
$$
Note that $\pi^{-1}(z) \cap \widehat{K} = \pi^{-1}(z) \cap K$ for every $z \in Q$ because $K$ is compact, $K \subset \pi^{-1}(\partial\Delta^{\prime}(r))$ and $\pi^{-1}(z) \cap K$ is discrete for every $z \in Q$.
By Lemma $\ref{A1}$, $\widehat{K}\cap(\pi^{-1}(\Delta^{\prime}(r)))$ is an analytic cover with $s$-sheets of $\pi^{-1}(\Delta^{\prime}(r))$. Hence, $V \equiv \widehat{K}\cap(\Delta^{\prime}(r)\times \Delta^{\prime \prime}(\delta))$ is also an analytic cover with at most $s$-sheets of $\Delta$. Thus, $V$ is an 1-dimensional subvariety of $\Delta$ such that $z \in V$ for each $z \in \mbox{spt}\langle T(r) , \pi , \alpha \rangle$. This completes the proof of this lemma.

\end{proof}

Finally, we can apply the argument in the first and second paragraph of \cite[pg 135]{A97} to conclude our main theorem for $k = 1$ since Lemma \ref{1} and Lemma \ref{nonzero} that Alexander applied have their counterparts for locally real rectifiable currents.

\section{Proof of the main theorem}
Now suppose that Theorem \ref{main theorem} is true for $k - 1$ where $k \geq 2$. We will apply Lemma \ref{slice} and Theorem \ref{area formula} to reduce the condition of Theorem \ref{main theorem} to the condition that $T$ is positive, and complete the proof by \cite[Theorem 3.9]{H77}. The same argument can be applied in \cite{A97} with \cite[Theorem 3.9]{H77} replaced by \cite[Theorem 5.2.1]{HL75}, but our method is simpler. So this simplifies the proof of Alexander. The main formula that enables us to do induction is the following result
which is from \cite[4.3.2 (1)]{F69}.

\begin{lemma}\label{slice}
Let $W$ be an open set in $\mathbb{R}^m$ and let $R \in F^{loc}_l(W)$. Suppose that $f : W \rightarrow \mathbb{R}^n$ is a smooth map. Then $\langle R , f , a \rangle$ exists for $\mathcal{L}^n$-a.e. $a \in \mathbb{R}^n$, and
$$
\int_{\mathbb{R}^n} \langle R , f , a \rangle (\varphi) d\mathcal{L}^n = (R\wedge f^{\ast}\omega_n)(\varphi)
$$
for all $\varphi \in A^{m-n-l}_c(W)$ where $\omega_n = dx_1\wedge ... \wedge dx_n$.
\end{lemma}

Now we can complete the induction of the proof of our main result Theorem \ref{main theorem}.

\begin{proof}
Consider $k\geq 2$. Recall that $T \in RR^{loc}_{k,k}(U)$, $T$ is $d$-closed and $\mbox{spt}(T)$ is $\mathcal{H}^{2k}$-locally finite. Choose $r > 0$ such that $\Delta(0 , r) \subset U$, and restrict $T$ to $\Delta(0 , r)$. We only need to show that $T|_{\Delta(0, r)} \in RR^{loc}_{k,k}(\Delta(0, r))$ is a real holomorphic $k$-chain. So we may assume $U = \Delta(0, r)$. Associate $T$ the oriented real $2k$-rectifold $(W , \theta , \vec{T})$ (see \cite[Definition 2.8]{JC}).
Let $T^{\prime} \in RR^{loc}_{k,k}(U)$ be the current given by $T^{\prime}(\varphi) = \int_W \langle \varphi , \xi \rangle d\mathcal{H}^{2k}$ for $2k$-forms $\varphi$, where $\xi(z) = \pm \theta(z)\vec{T}$ is a simple $2k$-vector which represents the naturally oriented complex tangent plane to $W$ at $z$ for $\mathcal{H}^{2k}$-a.e. $z \in W$, the sign $\pm$ being chosen so that $T^{\prime}$ is a positive current. We will show that $T^{\prime}$ is $d$-closed, then the result follows by \cite[Theorem 3.9]{JC}.

First, for each $j, \ 1 \leq j \leq n$, by \cite[pg 437]{F69},
$$
d\langle T , \pi_j , a \rangle = \langle dT , \pi_j , a \rangle = 0 \ \mbox{ and spt}\langle T , \pi_j , a \rangle \subset \mbox{spt}(T) \cap \pi_j^{-1}(a)
$$
for almost all $a \in B_r(0) \subset \mathbb{C}$. By Theorem \ref{volume},
$$
\int_{B_r(0)}^{\ast} \mathcal{H}^{2k-2}(\mbox{spt}(T) \cap \pi_j^{-1}(a)) d\mathcal{H}^2(a) \leq \frac{\Omega(2k-2)\Omega(2)}{\Omega(2k)}\mathcal{H}^{2k}(\mbox{spt}(T)) < \infty.
$$
Hence $\mbox{spt}(T) \cap \pi_j^{-1}(a)$ is $\mathcal{H}^{2k-2}$-locally finite for almost all $a \in B_r(0)$. By the induction hypothesis, $\langle T , \pi_j , a \rangle$ is a real holomorphic $(k - 1)$-chain for almost all $a \in B_r(0)$.
Let $\langle T , \pi_j , a \rangle^{\prime}$ be the positive $(k-1,k-1)$-current associated to $\langle T , \pi_j , a \rangle$. By \cite[4.3.8]{F69}, $\langle T , \pi_j , a \rangle^{\prime} = \langle T^{\prime} , \pi_j , a \rangle$ for almost all $a$. This implies that $\langle T^{\prime} , \pi_j , a \rangle$ is a positive real holomorphic $(k - 1)$-chain for almost all $a$.

Second, to check that $dT^{\prime}(\varphi) = 0$ for all $\varphi \in A^{2k-1}_c(U)$, it suffices to consider those forms of types $(k-1,k)$ and $(k,k-1)$ because $T^{\prime}$ is of bidimension $(k,k)$. We prove the case $(k-1,k)$ since the other case is similar. Let $\varphi = 4fdz_{i_1}\wedge \cdots \wedge dz_{i_{k-1}}\wedge d\Bar{z}_{j_1}\wedge \cdots \wedge d\Bar{z}_{j_k}$. Observe that
\begin{align*}
4dz_i\wedge d\Bar{z}_j = & \ (dz_i + dz_j)\wedge\overline{(dz_i + dz_j)} - (dz_i - dz_j)\wedge\overline{(dz_i - dz_j)}\\
& +i(dz_i + idz_j)\wedge\overline{(dz_i + idz_j)} - i(dz_i - idz_j)\wedge\overline{(dz_i - idz_j)}.
\end{align*}
By the above observation, we can factor $\varphi$ into four components :
\begin{align*}
(-1)^{k-2}\varphi = & \ f(dz_{i_1} + dz_{j_1})\wedge\overline{(dz_{i_1} + dz_{j_1})}\wedge dz_{i_2}\wedge \cdots \wedge dz_{i_{k-1}}\wedge d\Bar{z}_{j_2}\wedge \cdots \wedge d\Bar{z}_{j_k} \\
& - f(dz_{i_1} - dz_{j_1})\wedge\overline{(dz_{i_1} - dz_{j_1})}\wedge dz_{i_2}\wedge \cdots \wedge dz_{i_{k-1}}\wedge d\Bar{z}_{j_2}\wedge \cdots \wedge d\Bar{z}_{j_k}\\
& + if(dz_{i_1} + idz_{j_1})\wedge\overline{(dz_{i_1} + idz_{j_1})}\wedge dz_{i_2}\wedge \cdots \wedge dz_{i_{k-1}}\wedge d\Bar{z}_{j_2}\wedge \cdots \wedge d\Bar{z}_{j_k}\\
& - if(dz_{i_1} - idz_{j_1})\wedge\overline{(dz_{i_1} - idz_{j_1})}\wedge dz_{i_2}\wedge \cdots \wedge dz_{i_{k-1}}\wedge d\Bar{z}_{j_2}\wedge \cdots \wedge d\Bar{z}_{j_k}.
\end{align*}
Therefore, by change of variables, we can further assume that
$$
\varphi = \frac{i}{2} f(dz_{i_1}\wedge d\Bar{z}_{i_1})\wedge dz_{i_2}\wedge  \wedge dz_{i_{k-1}}\wedge d\Bar{z}_{j_2}\wedge \cdots \wedge d\Bar{z}_{j_k}
= \omega_{i_1}\wedge \psi,
$$
where $\psi = fdz_{i_2}\wedge \cdots \wedge dz_{i_{k-1}}\wedge d\Bar{z}_{j_2}\wedge \cdots \wedge d\Bar{z}_{j_k}$.
Then
$$
dT^{\prime}(\varphi) = (-1)^{2k+1}T^{\prime}(d\varphi) = -T^{\prime}(\omega_{i_1}\wedge d\psi) =  -(T^{\prime}\wedge \pi_{i_1}^{\ast}\omega_{i_1})(d\psi)
$$
By Lemma \ref{slice}, we have
$$
dT^{\prime}(\varphi) = \int_{B_r(0)} -\langle T^{\prime} , \pi_{i_1} , a \rangle (d\psi) d\mathcal{L}^2(a) = \int_{B_r(0)} d\langle T^{\prime} , \pi_{i_1} , a \rangle (\psi) d\mathcal{L}^2(a) = 0
$$
since $\langle T^{\prime} , \pi_{i_1} , a \rangle $ is a positive real holomorphic $(k - 1)$-chain and hence $d$-closed for almost all $a$.
Therefore by our result in \cite{JC}, $T'$ and hence $T$ are real holomorphic $k$-chains.
\end{proof}

\section{Applications}
In this section, we are going to generalize some results in $\cite[Section\ 4]{JC}$ and $\cite{HS74}$. We first give a generalization of \cite[Proposition 4.1]{JC}
to get rid of the positivity condition on $e$.

\begin{proposition}\label{R+dd^c}
Let $X$ be a complex projective manifold of complex dimension $n$ and $e\in A^{n-k, n-k}(X)$ be a $d$-closed form.
If $e$ considered as a current can be written as
$$e=R+dd^cb$$
where $R$ is a current such that the $(k, k)$-part $R_{k,k}$ of $R$ has finite mass and $\mathcal{H}^{2k}$-locally finite support, then $e$ is homologous to some algebraic cycle with real coefficients.
\end{proposition}

\begin{proof}
This is a generalization of \cite[Proposition 4.1]{JC}. In the original proof, we need to additionally assume that $R_{k,k}$ is positive to assert that $R_{k,k}$ is a real holomorphic chain. But by Theorem \ref{main theorem}, we can directly conclude that $R_{k,k}$ is a real holomorphic chain without the positivity on $R_{k,k}$.
\end{proof}

\begin{corollary}
Let $X$ be a complex projective manifold of dimension $n$. Given a smooth $d$-closed form $e\in A^{n-k, n-k}(X)$.
If $e$ is homologous to a Lipschitz $2k$-chain $P$  with rational coefficients which is $d^c$-closed,
then $e$ is homologous to an algebraic cycle with rational coefficients.
\end{corollary}

\begin{proof}
By assumption, we have
$$e=P+da$$
for some $a\in \mathscr{D}'_{2(k+1)}(X)$. Then $d^ce = 0 = d^cP + d^cda = d^cda$. So
$$
e = P + dd^cb
$$
for some $b$ by the $dd^c$-lemma. Clearly, $\mbox{spt}(P_{k,k}) \subset \mbox{spt}(P)$ and $\M(P_{k,k}) \leq \M(P) < \infty$ where $P_{k,k}$ is the $(k,k)$-part of $P$. Hence by Proposition \ref{R+dd^c} and the rationality of $e$, $e$ is homologous to an algebraic cycle with rational coefficients.
\end{proof}

Recall that $N^{loc}_k(M)$ denotes the group of locally normal $k$-currents on $M$.

\begin{corollary}\label{characterization of RR 2}
Let $M$ be a smooth manifold and $T \in N^{loc}_k(M)$. If there is a constant $c > 0$ such that $\Theta^k(\|T\| , a) \geq c$ for $\mathcal{H}^k$-a.e. $a \in \mbox{spt}(T)$, then $\mbox{spt}(T)$ is $\mathcal{H}^k$-locally finite.
\end{corollary}

\begin{proof}
Let $E = \{ a \in \mbox{spt}(T) : \Theta^k(\|T\| , a)<c \}$. Then $\mathcal{H}^k(E) = 0$. Since $T$ is locally flat, by \cite[4.2.14]{F69}, $\|T\|(E) = 0$. Let $K\subset M$ be any compact subset. Since $T$ is locally normal, $||T||(K)<\infty$.
By \cite[2.10.19(3)]{F69},
$$||T||(K)\geq c\mathscr{S}^k(\mbox{spt}(T)\cap K-E)\geq c\mathcal{H}^k(\mbox{spt}(T)\cap K-E)=c\mathcal{H}^k(\mbox{spt}(T)\cap K)$$
which implies that $\mathcal{H}^k(\mbox{spt}(T)\cap K)<\infty$ and hence $T$ is $\mathcal{H}^k$-locally finite.
\end{proof}

\begin{theorem}
Let $X$ be a complex manifold and $T \in N^{loc}_{k,k}(X)$ be $d$-closed.
\begin{enumerate}
\item
If there is a constant $c > 0$ such that $\Theta^{2k}(\|T\|, a) \geq c$ for $\mathcal{H}^{2k}$-a.e. $a \in \mbox{spt}(T)$, then $T$ is a real holomorphic $k$-chain.
\item
If $N = \{ a \in X : \Theta^{2k}(\|T\| , a) > 0 \}$ is $\mathcal{H}^{2k}$-locally finite and $\Theta^{2k}(\|T\|, a)>0$ for
$\mathcal{H}^{2k}$-a.e. $a\in \mbox{spt}(T)$, then $T$ is a real holomorphic $k$-chain.
\end{enumerate}
\end{theorem}

\begin{proof}
1. By Corollary \ref{characterization of RR 2}, $T \in RR^{loc}_{k,k}(X)$ and $\mbox{spt}(T)$ is $\mathcal{H}^{2k}$-locally finite.
      Then by Theorem \ref{main theorem}, $T$ is a real holomorphic $k$-chain.

2. Note that $N \subset \mbox{spt}(T)$ and $\mathcal{H}^{2k}(\mbox{spt}(T)-N) = 0$. For any compact set $K\subset X$,
$$\mathcal{H}^{2k}(\mbox{spt}(T)\cap K) = \mathcal{H}^{2k}((\mbox{spt}(T)-N)\cap K) + \mathcal{H}^{2k}(\mbox{spt}(T)\cap N \cap K)\leq \mathcal{H}^{2k}(N\cap K) < \infty$$
which implies that $T$ has $\mathcal{H}^{2k}$-locally finite support. By \cite[Theorem 32.1]{S}, $T\in RR^{loc}_{k, k}(X)$ and by Theorem \ref{main theorem}, $T$ is a real holomorphic $k$-chain.

\end{proof}

Harvey and King proved a structure theorem for positive currents in $\cite{HK}$ :

\begin{theorem}\label{structure theorem}
Let $U\subset \C^n$ be an open set. Suppose that $u \in \mathscr{D}^{\prime}_{k,k}(U)$ is positive and d-closed. Assume that for each compact set $K \subset U$ there exists a constant $c > 0$ such that $n(u , a) \geq c$ for all $a \in \mbox{spt}(u)\cap K$. Then there exists a pure $2k$-dimensional subvarieties $V$ of $U$ and positive real numbers $a_j$ for each irreducible component $V_j$ of $V$ such that $u = \sum^{\infty}_{j=1} a_j[V_j]$.
\end{theorem}

We give an analogous but more general result.

\begin{theorem}
Suppose $T \in N^{loc}_{k,k}(U)$ and $dT = 0$. Assume that for each compact set $K \subset U$ there exists a constant $c > 0$ such that $\Theta^{\ast 2k}(\|T\| , a) \geq c$ for all $a \in \mbox{spt}(T)\cap K$. Then $T$ is a real holomorphic $k$-chain.
\end{theorem}

\begin{proof}
Clearly, $\Theta^{\ast 2k}(\|T\| , a) = 0$ for all $a \in U- \mbox{spt}(T)$. Hence $\Theta^{\ast 2k}(\|T\| , a) > 0$ for $\|T\|$-a.e. $a \in U$. By \cite[Theorem 32.1]{S}, $T$ is real rectifiable. Since for any fixed $a \in \mbox{spt}(T)$ and given any $r > 0$, there is a $c > 0$ such that $\Theta^{2k}(\|T\| , b) \geq c$ for $\mathcal{H}^{2k}$-a.e. $b \in \mbox{spt}(T)\cap B_r(a)$.  Therefore
$$\infty > \|T\|(B_r(a)\cap \mbox{spt}(T)) \geq c\mathcal{H}^{2k}(\mbox{spt}(T)\cap B_r(a))$$
which implies that $\mbox{spt}(T)$ is $\mathcal{H}^{2k}$-locally finite. The result follows by Theorem \ref{main theorem}.
\end{proof}

Let $M$ be an oriented Riemannian manifold and $V$ be a compactly supported vector field on $M$. Consider the flow $\{h_t\}$ for $V$ where $\{h_t\}$ is a 1-parameter group of diffeomorphisms of $M$ with $h_t(x) = x$ for $x \notin \mbox{supp}(V)$. Let $W$ be a relatively compact open subset of $M$ such that $\mbox{supp}(V) \subset W$. For $T \in RR^{loc}_{k}(X)$, let
$$ J_{T,V}(t) = \|h_{t\ast}T\|(W)$$
and
$$\delta^{(j)}(T , V) = J_{T,V}^{(j)}(0)$$
the $j^{th}$ derivative of $J_{T,V}(t)$.

Note that $\delta^{(j)}(T , V)$ does not depend on the choice of $W$ since $h_t(x) = x$ for $x \notin \mbox{supp}(V)$. $\delta^{(j)}(T , V)$ is called the $j^{th}$ variation of $T$ with respect to $V$.

\begin{definition}
\begin{enumerate}
\item
We say that a current $T \in RR^{loc}_k(M)$ is stationary if $\delta^{(1)}(T , V) = 0$ for all compactly supported vector fields $V$ on $M$.
\item
We say that a current $T \in RR^{loc}_k(M)$ is stable if $J_{T,V}(t)$ has a local minimum at $t = 0$ for all compactly supported vector fields $V$ on $M$.
\end{enumerate}
\end{definition}

\begin{remark}
If $T$ is stable, then $T$ is stationary and $\delta^{(2)}(T , V) \geq 0$.
\end{remark}

By \cite[16.1, 16.2, 16.3, 17.8]{S}, we have the following theorem :

\begin{theorem}\label{stationary}
If  $T \in RR^{loc}_k(M)$ is stationary, then the density  $\Theta^k(\|T\| , x) = \lim_{r\rightarrow 0} \frac{\|T\|(B_r(x))}{\Omega(k)r^k}$ exists at every point $x \in U$, and $\Theta^k(\|T\| , \bullet)$ is an upper-semi-continuous function in $U$ :
$$
\Theta^k(\|T\| , x) \geq \limsup_{y\rightarrow x} \Theta^k(\|T\| , y) \ \ \ \ \   \forall \ x \in U.
$$
\end{theorem}

\begin{corollary}\label{stationary2}
Let $X$ be a complex manifold.
Suppose that $T \in RR^{loc}_{k,k}(X)$ is stationary and $dT = 0$. If there is a constant $c > 0$ such that $\Theta^k(\|T\| , x)$ is either equal to 0 or larger than $c$, then $T$ is a real holomorphic $k$-chain.
\end{corollary}

\begin{proof}
By Theorem \ref{stationary}, $N = \{x \in M : \Theta^k(\|T\| , x) > 0\}$ is closed. By \cite[Proposition 3.7]{JC}, we have $\mbox{spt}(T) = N$. By \cite[Theorem 3.2(1)]{S}, for each compact subset $K \subset U$, $c\mathcal{H}^{2k}(K\cap \mbox{spt}(T)) \leq \|T\|(K\cap \mbox{spt}(T)) < \infty$. This implies that $\mbox{spt}(T)$ is $\mathcal{H}^{2k}$-locally finite. It then follows from Theorem \ref{main theorem} that $T$ is a real holomorphic $k$-chain.

\end{proof}

For $M$ a compact oriented Riemannian manifold, Federer and Fleming (\cite{FF60}) showed that the integral homology groups $H_{\ast}(M , \mathbb{Z})$ are naturally isomorphic to the homology groups of the chain complex $I_{\ast}(M)$ with boundary map $d$. By a simple modification of their proof, we can show that the real homology groups $H_{\ast}(M , \mathbb{R})$ is isomorphic to the homology groups of the chain complex of realistic currents $RE_{\ast}(M)$ which are defined as follows.

\begin{definition}
Let $M$ be a smooth manifold. We say that $T\in \mathscr{D}'_k(M)$ is realistic(resp. locally realistic) if
$T\in RR_k(M)$ and $dT\in RR_{k-1}(M)$ (resp. $T\in RR^{loc}_k(M)$ and $dT\in RR^{loc}_{k-1}(M)$).
\end{definition}

We give real counterparts of homologically volume minimizing currents, stationary currents and stable currents that appeared
in Section 3 of \cite{HS74}.

\begin{definition}
A current $T \in RR^{loc}_k(M)$ is said to be real homologically volume minimizing if
$$\M(T) \leq \M(T + dS)$$
for all $S \in RR^{loc}_{k+1}(M)$.
\end{definition}

For a complex manifold $X$, we denote by $\mathscr{RZ}^+_k(X), \mathscr{RZ}^-_k(X)$ the collections of positive and negative real holomorphic $k$-chains on $X$ respectively. The main tool we are going to use is the Wirtinger's inequality which says that if $X$ is a compact K\"ahler manifold with K\"ahler form $\omega$, then for $S \in RR^{loc}_{k,k}(X)$,
$$
S(\frac{1}{k!}\omega^k) \leq \M_X(S)
$$
with equality holding if and only if the tangent $2k$-vectors to $S$ are complex and positive $\|S\|$-almost everywhere (see \cite[5.4.19]{F69}).
The following result is a generalization of \cite[Proposition 3.1]{HS74}.

\begin{proposition}\label{property of holomorphic chain}
Let $X$ be a compact K\"ahler manifold.
\begin{enumerate}
\item
If $S \in \mathscr{RZ}^+_k(X)$, then $S$ is real homologically minimizing.
\item
If $S \in \mathscr{RZ}_k(X)$, then $S$ is stable.
\end{enumerate}
\end{proposition}

\begin{proof}
(a) For any $R \in RE^{loc}_{k+1}(X)$,
$$
\M(S + dR) \geq (S + dR)(\frac{1}{k!}\omega^k) = S(\frac{1}{k!}\omega^k) = \M(S),
$$
since $\frac{1}{k!}\omega^k$ is $d$-closed.

(b) Suppose $S = \sum^{\infty}_{j=1} r_j [V_j]$ with $r_j \in \mathbb{R}$. Then $S$ can be expressed as $S_1 - S_2$ where $S_1$ and $S_2$ belong to $\mathscr{RZ}^+_k(X)$. Let $\{h_t\}$ be a 1-parameter family of diffeomorphisms of $X$ and $A = \mbox{spt}(S_1)\cap \mbox{spt}(S_2)$. Since $A$ is a holomorphic subvariety of complex dimension $(k-1)$, $A$ has $\|S_1\|$ and $\|S_2\|$ measure zero. Hence $h_t$ is a diffeomorphism implies that $\M(h_{t\ast}S) = \M(h_{t\ast}S_1) + \M(h_{t\ast}S_2)$. Let $H : h_0 \simeq h_y$ be the deformation from $h_0$ to $h_t$. By the homotopy formula, we have
$$
S_i - h_{t\ast}S_i = (-1)^k dH_{\ast}(I_t \times S_i), i = 1, 2,
$$
where $I_t = [0,t]$ or $[t, 0]$. Hence, by part (a), $\M(S_i) \leq \M(h_{t\ast}S_i), i = 1, 2$. Therefore,
$$
\M(S) = \M(S_1) + \M(S_2) \leq \M(h_{t\ast}S_1) + \M(h_{t\ast}S_2) = \M(h_{t\ast}S).
$$
This shows that $\M(h_{t\ast}S)$ has a minimum at $t = 0$ and thus $S$ is stable.
\end{proof}

The following is a generalization of \cite[Theorem 3.2]{HS74}.

\begin{theorem}\label{homological class represented by chain}
Let $X$ be a compact K\"ahler manifold. Suppose that $\gamma \in H_{2k}(X , \mathbb{R})$ has a representative $S \in \mathscr{RZ}^+_k(X)$. If $\gamma$ has a homologically volume minimizing representative $T \in RR_{2k}(X)$ such that $N = \{x \in X : \Theta^{2k}(\|T\| , x) > 0 \}$ is $\mathcal{H}^{2k}$-locally finite, then $T \in \mathscr{RZ}^+_k(X)$.
\end{theorem}

\begin{proof}
Since both $S$ and $T$ belong to $\gamma$, there exists a realistic current $R \in RE_{2k}(X)$ such that $S = T + dR$. Hence
$$
\M(S) = S(\frac{1}{k!}\omega^k) = (T + dR)(\frac{1}{k!}\omega^k) = T(\frac{1}{k!}\omega^k) \leq \M(T).
$$
By assumption, $\M(S) = \M(T)$. Therefore $\M(T) = T(\frac{1}{k!}\omega^k)$ which implies $T$ is positive. By \cite[Theorem 3.9]{JC}, $T$ is a positive real holomorphic $k$-chain.
\end{proof}

\begin{corollary}
Let $X$ be a compact K\"ahler manifold. Suppose that $\gamma \in H_{2k}(X , \mathbb{R})$ has a representative $S \in \mathscr{RZ}^+_k(X)$. If $\gamma$ has a homologically volume minimizing representative $T \in RR_{2k}(X)$ such that $\mbox{spt}(T)$ is $\mathcal{H}^{2k}$-locally finite, then $T \in \mathscr{RZ}^+_k(X)$.
\end{corollary}

If $X = \mathbb{C}P^n$ (complex projective $n$-space with the usual
K\"ahler metric), then each $2k$-dimensional integral homology class has a representative either in $\mathscr{Z}^+_k(X)$ or $\mathscr{Z}^-_k(X)$. So we also have each $2k$-dimensional real homology class has a representative either in $\mathscr{RZ}^+_k(X)$ or $\mathscr{RZ}^-_k(X)$.
The integral case of the following result was obtained by Harvey and Shiffman in \cite[Corollary 3.3]{HS74}.

\begin{corollary}
 Let $T \in RR_{2k}(\mathbb{C}P^n)$ with $dT = 0$. Suppose that $\mathcal{H}^{2k}(\mbox{spt}(T)) < \infty$ or there is a constant $c > 0$ such that $\Theta^{2k}(\|T\| , x) \geq c$ for $\|T\|$-almost all $x$. Then $T$ is real homologically volume minimizing if and only if either $T \in \mathscr{RZ}^+_k(X)$ or $\mathscr{RZ}^-_k(X)$.
\end{corollary}

\begin{proof}
For the case $\mathcal{H}^{2k}(\mbox{spt}(T)) < \infty$, the result follows from Theorem \ref{homological class represented by chain}. Now suppose there is a constant $c > 0$ such that $\Theta^{2k}(\|T\| , x) \geq c$ for $\|T\|$-almost all $x$. By the proof of \cite[7.1.7]{F69}, we can show that every closed real homologically minimizing current on a compact oriented Riemannian manifold is stable. Then the conclusion follows by Corollary \ref{stationary2}. The converse follows by Proposition \ref{property of holomorphic chain}.
\end{proof}

\begin{theorem}
Let $\{T_j\}^{\infty}_{j=1}$ be a sequence of real holomorphic $k$-chains on $X$ with locally uniformly bounded mass (see \cite[Theorem 3.9]{HS74}). Suppose that $\cup_{j=1}^{\infty} \mbox{spt}(T_j)$ is $\mathcal{H}^{2k}$-locally finite. Then there exists a subsequence of $\{T_j\}^{\infty}_{j=1}$ that converges in the locally flat topology to a real holomorphic $k$-chain.
\end{theorem}

\begin{proof}
By \cite[4.2.17(1)]{F69}, there exist $T \in N^{loc}_{2k}(X)$ and a subsequence $\{T_{j_k}\}^{\infty}_{j=1}$ such that $T_{j_k}\rightarrow T$ in the locally flat topology. Since each $T_j$ is of bidimension $(k,k)$, $T \in N^{loc}_{k,k}(X)$. By assumption and \cite[Proposition 3.8]{JC}, $\cup^{\infty}_{j=1} \mbox{spt}(T_j)$ is closed in $X$. So we have $\mbox{spt}(T) \subset \overline{\cup^{\infty}_{j=1} \mbox{spt}(T_j)} = \cup^{\infty}_{j=1} \mbox{spt}(T_j)$. By \cite[Corollary 2.14]{JC}, $T \in RR^{loc}_{k,k}(X)$. Clearly, $T$ is $d$-closed. Thus, by Theorem \ref{main theorem}, $T$ is a real holomorphic $k$-chain.
\end{proof}

\begin{remark}
The condition that $\cup^{\infty}_{j=1} \mbox{spt}(T_j)$ is $\mathcal{H}^{2k}$-locally finite is necessary. Take $T_j = \sum_{n=1}^j \frac{1}{n^2}[\frac{1}{n}]$ and $T = \sum_{n=1}^{\infty} \frac{1}{n^2}[\frac{1}{n}]$ in $\mathbb{C}$. Then $T_j$ converges to $T$ in mass norm, but $T$ is not a holomorphic chain.
\end{remark}

\end{document}